\documentclass{amsart}
\usepackage[margin=1.2in]{geometry}

\newtheorem{theorem}{Theorem}[section]
\newtheorem{lemma}[theorem]{Lemma}

\theoremstyle{definition}
\newtheorem{definition}[theorem]{Definition}

\theoremstyle{remark}

\numberwithin{equation}{section}

\usepackage{biblatex}
\begin{document}

\title{An Asymptotic Determinant Bound on 0-1 Matrices with Fixed Row Sums}

\author{Justin Semonsen}
\address{Department of Mathematics, Rutgers University, Piscataway, New Jersey 08854}
\curraddr{Department of Mathematics, Rutgers University, Piscataway, New Jersey 08854}
\email{js2118@math.rutgers.edu}
\thanks{Supported by the Simons  Collaboration on Algorithms and Geometry  through Simons Foundation award 332622 and Swastik Kopparty's NSF grant CCF-1814409.}

\date{February 4, 2020}


\begin{abstract}
This paper improves previously known bounds on the determinant of 0-1 matrices where each row has fixed support size. This uses a method based on Scheinerman's, with new analyses to improve upon his conjectures.

\end{abstract}
\maketitle
\vspace{-.5cm}
\section{Determinant Bounds}
Hadamard's maximum determinant problem asks for the largest determinant among all $n\times n$ zero-one matrices \cite{Had_1893}. This problem has been well studied, but many questions still remain unanswered.

For the remainder of this paper, let $\vec{A}_i$ be the $i$th row of the matrix $A$. Also let $\Vert \cdot \Vert$ represent the standard $l_2$ norm on vectors.

We will look at the maximum determinant of the restricted class of zero-one matrices defined below:

\begin{definition}
$R(n,k) = \{A \in M_{n \times n}(\{0,1\}): \Vert \vec{A}_i\Vert = \sqrt{k}\,\forall\,1 \leq i \leq n\}$ for $1 \leq k \leq n$.
\end{definition}


We can also characterize the matrices in $R(n,k)$ as the matrices in $M_{n \times n}(\{0,1\})$ whose rows each sum to $k$. This means the vector of all ones is a left eigenvector with eigenvalue $k$.

The question originally posed by Scheinerman \cite{Scheinerman:2019aa} is what is the largest determinant that can be attained in $R(n,k)$? This lets us define the following quantity:


\begin{definition} $M(n,k) = \max_{A \in R(n,k)} \vert \det(A) \vert$ \end{definition}




The first upper bound on $M(n,k)$ that we can give is due to Hadamard \cite{Had_1893}: This says that $\vert \det(A) \vert \leq \prod_{i=1}^n \Vert \vec{A}_i \Vert$. Since $\Vert \vec{A}_i \Vert = \sqrt{k}$ for every $i$, this means $\vert \det(A) \vert \leq k^\frac{n}{2}$. Equivalently, we can write $M(n,k)^2 \leq k^n$, as the bounds on the square of the determinant are often more concise.







A more general result of Ryser \cite{Ryser_1956} gives that $M(n,k) \leq k(k-\lambda)^\frac{n-1}{2}$ where $\lambda = \frac{k(k-1)}{n-1}$. When $k$ is large relative to $n$, this is an exponential improvement on Hadamard, but Ryser showed this bound can't be tight unless $\lambda$ is integral.

This means that if $k$ is fixed and $n$ grows to infinity, not only is this bound not tight, but $\lambda \to 0$ and Ryser's bound has the same exponential growth as Hadamard's bound.

Bruhn and Rautenbach \cite{Bruhn_2018} gave an exponentially better bound for $k=2$, which gives $M(n,k) \leq \left(\sqrt[3]{2}\right)^n$. They could only conjecture that a similar exponential improvement was possible for $k=3$ and larger. In other words, until Scheinerman \cite{Scheinerman:2019aa} it was only known that $\limsup_{n \to \infty} M(n,k)^\frac{1}{n} \leq \sqrt{k}$.

We note that if we have an $m \times m$ matrix $A$ for which $\det(A) = c^m$, then $\det(A \otimes I_t) = c^{mt}$ for any $t$, and thus $\limsup_{n \to \infty} M(n,k)^\frac{1}{n} \geq c$. A projective plane of order $k-1$ has an $(k^2 - k +1) \times (k^2 - k +1)$ incidence matrix $A$ with $k$ ones in each row, and thus is in $R(k^2 - k +1, k)$. Since $AA^\top  = J + (k-1) I$, we have that $\det(A)^2 = \det(J_{k^2-k+1} + (k-1)I_{k^2-k+1}) = k^2(k-1)^{k^2 - k}$, so $\limsup_{n \to \infty} M(n,k)^\frac{1}{n} \geq k^\frac{1}{k^2 - k +1}(k-1)^\frac{k^2 - k}{2k^2 - 2k +2} = \sqrt{k} - \frac{1}{2\sqrt{k}} + O(k^\frac{-3}{2})$.

Scheinerman \cite{Scheinerman:2019aa} was able to close this gap slightly by decomposing the matrix into blocks of rows, then analyzing each block separately:

\begin{theorem}Let $q$ be an integer with $1 \leq q \leq k$. Then $M(n,k) \leq c_{q,k}^n$ for $$c_{q,k} = (q+k-1)^{\frac{1}{2q}\left(1 -\frac{q-1}{k}\right)}(k-1)^{\frac{q-1}{2q}\left(1 -\frac{q-1}{k}\right)} k^\frac{q-1}{2k}$$
\end{theorem}

For $q =1$, this matches the Hadamard bound, but when $q > 1$, we have that $c_{q,k} < \sqrt{k}$. In addition, Scheinerman showed that for some sequence $q_k$, $c_{q_k,k} = \sqrt{k} - \frac{t_1}{2\sqrt{k}} + O(k^\frac{-3}{2})$ for a constant $t_1 \approx 0.096$.

Scheinerman also proposed an algorithm for combining his methods to provide an even tighter bound, but his analysis only provided an improvement for $k \leq 27$. In this work, we provide a tighter analysis of this algorithm, showing that it improves on Scheinerman's other techniques for bounding the determinant, and analyze the asymptotic behavior of this bound.

We also use similar techniques to improve upon Scheinerman's algorithm, and show that the resulting bound is an improvement upon the previous algorithm.

\section{Decomposition bounds}
Since these techniques depend heavily on the decomposition by Scheinerman in \cite{Scheinerman:2019aa}, here we will give a simple overview of the needed quantities and definitions as they appear in this paper.

For any $m \times n$ matrix $B$ we define $Vol(B) = \sqrt{\vert \det(BB^\top) \vert}$. Since the inner matrix product is the Gram matrix of the rows, this quantity is essentially the $m$-dimensional volume of the box with sides given by the rows of $B$.

This measure has two useful properties: First, if $A$ is an $n \times n$ square matrix, then $Vol(A) = \sqrt{\vert \det(AA^\top) \vert} = \sqrt{\det(A)^2} = \vert \det(A) \vert$. Secondly if $B_1$ and $B_2$ are $m_1 \times n$ and $m_2 \times n$ matrices respectively, then let $B = \left[\begin{smallmatrix} B_1\\B_2\end{smallmatrix}\right]$, the $(m_1 + m_2) \times n$ block matrix with $B_1$ above $B_2$. We can see that $BB^\top = \left[\begin{smallmatrix} B_1B_1^\top & B_1B_2^\top\\B_2B_1^\top& B_2B_2^\top\end{smallmatrix}\right]$, so by Fischer's Inequality:

\begin{equation}
\begin{split}
Vol(B) &= \sqrt{\left\vert\det\left(\left[\begin{array}{cc} B_1B_1^\top & B_1B_2^\top\\B_2B_1^\top& B_2B_2^\top\end{array}\right]\right) \right\vert}\\
&\leq \sqrt{\left\vert\det\left(\left[\begin{array}{cc} B_1B_1^\top & 0 \\ 0 & B_2B_2^\top\end{array}\right]\right) \right\vert} = Vol(B_1)Vol(B_2)
\end{split}
\end{equation}

For convenience of notation we will sometimes use $M(n,k)^2$ instead of $M(n,k)$ in most calculations, as $M(n,k)^2 = \max_{A \in S_{n,k}} Vol(A)^2 = \max_{A \in S_{n,k}} \det(AA^\top)$.

Given a matrix $A \in R(n,k)$, we can decompose $A$ into its rows $\vec{A}_i$. Then by repeatedly using the submultiplicativity, we can see that $Vol(A)^2 \leq \prod_{i=1}^n Vol(\vec{A}_i)^2 = k^n$. This was already given by Hadamard, but other decompositions will yield better results.

Scheinerman noticed that if we decompose $A$ into blocks $A_i$ of rows such that each block has a column of all ones, then if $A_i$ is an $m_i \times n$ block, then $A_iA_i^\top$ is an $m_i$ by $m_i$ matrix with $k$ on the diagonal and strictly positive integers off of it. 

Using a result by Olkin \cite{Olkin_2014}, Scheinerman proved that $Vol(A_i)^2 \leq \det(J_{m_i} + (k-1)I_{m_i}) = (m_i+k-1)(k-1)^{m_i-1}$. Since the $A_i$ are a partition of the $n$ rows $\sum_i m_i = n$ and thus $Vol(A)^2 \leq \prod_i (m_i+k-1)(k-1)^{m_i-1} = (k-1)^n \prod_i \left(1+ \frac{m_i}{k-1}\right)$.

We can see that we have at most $n$ such partitions, so we can assume that there are exactly $n$, with some allowed to be of size 0. By Jensen's inequality balancing these sizes maximizes the product, but that is the same as making each row its own partition. We are hoping to find a smaller bound, and thus go for as large of blocks as we can.

With this in mind, Scheinerman proposed the following algorithm: At each step $i$, choose the column with the most 1s in it. Let the rows that have those ones be the block $A_i$, and repeat the same steps on the remaining rows to partition the entire matrix. This greedy algorithm gives us a method for finding a better decomposition of a matrix.

In order to provide an upper bound on this, Scheinerman noted that the average number of ones per column in an $m \times n$ matrix is $\frac{km}{n}$, so there is always a column with at least $\lceil\frac{km}{n}\rceil$ ones. While it might be possible to find a column with more ones, this is always guaranteed, and this allows us to bound the determinant for all $A \in R(n,k)$.

\subsection{Analysis of the greedy algorithm bound}
In this section, we use a new analysis technique to get the following result:

\begin{theorem}\label{old-main}
For every $k$, there is a $c_k$ such that $M(n,k) \leq c_k$ where $c_k = \sqrt{k} - \frac{t_2}{2\sqrt{k}} + O(k^\frac{-3}{2})$ for $t_2 = 1 - \frac{\pi^2}{12}$.
\end{theorem}

This theorem gives $t_2 \approx 0.178$, closing the gap towards the lower bound given by the projective plane.

To get this bound from greedy algorithm, we will set up a linear programming problem to help us analyze it: Fix a matrix $A$ in $S(n,k)$. Let $x_j$ be the number of blocks of size $j$ in the decomposition given by Scheinerman's greedy algorithm. This gives us $n$ variables, as the blocks can be of size 1 up to $n$.

Since we always take the largest block, we know that after we remove all the blocks of size $i$ or larger, the $m$ remaining rows can't have any columns with $i$ ones in it. This means the average number of ones $\frac{km}{n}$ must be no bigger than $i-1$, so we can calculate the number of rows remaining by $m = n - \sum_{j=i}^n jx_j$. This means that the greedy decomposition of any matrix $A$ must satisfy the constraints $\sum_{j=i}^n kjx_j \geq (k-i +1)n$ for every $i \in [n]$.

Since the $x_i$ give us a partition of the $n$ rows, we also have that $\sum_{j=1}^n jx_j = n$, and that $x_i \geq 0$ for every $i$.

\begin{lemma}
Every greedy algorithm decomposition of a matrix $A \in R(n,k)$ has $x_j$ blocks of size $j$, where $x_j$ is an integral feasible solution to the following LP:

\begin{equation}
\begin{array}{ll@{}ll}
 \sum\limits_{j=i}^n kjx_j \geq (k-i+1)n, \hspace{.5cm}& i=1 ,..., n\\
 \sum\limits_{j=1}^n jx_j = n\\
                                                x_{j} \geq 0 ,&j=1 ,..., n
\end{array}
\end{equation}

\end{lemma}

A given set of $x_i$ generates the determinant bound $Vol(A)^2 \leq (k-1)^n \prod_{j=1}^n \left(1+ \frac{j}{k-1}\right)^{x_j}$. In order to find a bound for all matrices, we want to find the largest this could be for any matrix using the greedy algorithm.

While this doesn't fit the LP framework, we can remove the $(k-1)^n$ term and take the logarithm to give that our objective to be maximized is $\sum_{j=1}^n x_j\ln\left(1+ \frac{j}{k-1}\right)$.

Notice that since the right hand side of each bound is a multiple of $n$, we can simply scale down the variables to $a_j = \frac{x_j}{n}$. In addition, since each $a_j$ is non-negative, we can see that all the constraints with $i > k$ are trivially satisfied, and can thus be dropped. Similarly, note that the equality constraint ensures that the $i = 1$ constraint is trivially satisfied.

Lastly, we note that relaxing the integrality constraint can only weaken the bound, so we are left with the following lemma:

\begin{lemma}\label{old-proof}
For any fixed $k$, let $\alpha$ be the optimal solution of the following LP:

\begin{equation}\label{old-primal}
\begin{array}{ll@{}ll}
\text{maximize}  & \alpha = \sum\limits_{j=1}^{n} \ln\left(1+ \frac{j}{k-1}\right)a_j &\\
\text{subject to}& \sum\limits_{j=i}^n kja_j \geq k-i+1, \hspace{.5cm}& i=2 ,..., k\\
& \sum\limits_{j=1}^n ja_j = 1\\
& a_{j} \geq 0 ,&j=1 ,..., n
\end{array}
\end{equation}
Then $M(n,k)^2 \leq \gamma_k^{n}$ where $\gamma_k = (k-1)e^\alpha$.
\end{lemma}

Since there are $k$ constraints and $n$ variables, we can define the dual linear program as follows:

\begin{equation}\label{old-dual}
\begin{array}{ll@{}ll}
\text{minimize}  & \beta = \sum\limits_{i=1}^{k} (k-i+1)b_i &\\
\text{subject to}& \sum\limits_{i=1}^j kj b_i \geq \ln\left(1+ \frac{j}{k-1}\right), \hspace{.5cm}& j=1,..., n\\
&b_i \leq 0 ,&i=2 ,..., k
\end{array}
\end{equation}

This formulation lets us give an explicit formula for the optimal $\alpha$:

\begin{lemma}\label{old-calculation}
The optimal solution to equation \ref{old-primal} is given by $a_j = \frac{1}{kj}$ for $1 \leq j \leq k$ and $a_j = 0$ otherwise. The optimal solution to equation \ref{old-dual} is given by $b_1 = \frac{\ln\left(1+ \frac{1}{k-1}\right)}{k}$ and $b_i =  \frac{\ln\left(1+ \frac{i}{k-1}\right)}{ik} -  \frac{\ln\left(1+ \frac{i-1}{k-1}\right)}{(i-1)k}$ for $2 \leq i \leq k$.
\end{lemma}

\begin{proof}
Since $a_j = 0$ for all $j > k$, we simply need to show that $\sum\limits_{j=i}^k kja_j \geq k-i+1$ for each $2 \leq i \leq k$ and $\sum\limits_{j=1}^k ja_j = 1$. However, since $kja_j = 1$, it is easy to see that $a_j$ are a feasible solution to equation \ref{old-primal}.

In fact, these $a_j$ are the solution given by making all constraints tight, and thus if we let $\vec{a} \in \mathbb{R}^k$ be the vector of the non-zero $a_j$, we see that $M\vec{a} = \vec{v}$ where $\vec{v}$ is the vector given by $v_i = k-i-1$ and $M$ is the matrix with $M_{ij} = kj$ when $j \geq i$ and 0 otherwise.

This means that $\vec{a} = M^{-1}\vec{v}$ and so if we let $\vec{c}$ be given by $c_i = \ln\left(1+ \frac{i}{k-1}\right)$, we have that $\alpha = \vec{c}^\top \vec{a} = \vec{c}^\top M^{-1}\vec{v}$.

As for the dual solution, we can use Jensen to verify that $b_i \leq 0$ for every $2 \leq i \leq k$. We can also see that the first $k$ constraints are actually tight. To verify that the remaining constraints hold, simply use that Jensen gives that $\frac{\ln\left(1+ \frac{j}{k-1}\right)}{j}$ is a decreasing sequence.

This means that these $b_j$ are a feasible solution, and furthermore the one given by making the first $k$ constraints tight. Thus if we let $\vec{b} \in \mathbb{R}^k$ be the vector of the $b_i$, we have that $M^\top \vec{b} = \vec{c}$. This means that $\vec{b} = (M^\top)^{-1} \vec{c}$, and thus $\beta = \vec{v}^\top\vec{b} = \vec{v}^\top(M^{-1})^\top \vec{c}$.

This means that $\alpha = \beta$, and thus by duality these sets of $a_j$ and $b_i$ are optimal solutions to their respective problems.
\end{proof}

This allows us to finally prove Theorem \ref{old-main} by combining Lemma \ref{old-proof} with Lemma \ref{old-calculation}:

\begin{proof}
By Lemma \ref{old-proof}, we have that $M(n,k)^2 \leq \gamma_k^n$, so $M(n,k) \leq c_k^n$ where $c_k = \sqrt{\gamma}$. By Lemma \ref{old-calculation}, we can write $\gamma$ and $c_k$ as follows:

$$\gamma = (k-1)e^{\left(\sum\limits_{j=1}^{k} \frac{\ln\left(1+ \frac{j}{k-1}\right)}{jk}\right)}$$
$$c_k = \sqrt{k-1}e^\frac{\sum\limits_{j=1}^{k} \frac{\ln\left(1+ \frac{j}{k-1}\right)}{j}}{2k}$$

Using $x = \frac{j}{k-1}$, we can manipulate a part of this equation to resemble a Riemann approximation of the integral $\int\limits_0^1 \frac{\ln(1+x)}{x} = \frac{\pi^2}{12}$. This gives that:

\begin{align*}
\sum\limits_{j=1}^{k} \frac{\ln\left(1+ \frac{j}{k-1}\right)}{j} &=  \frac{\ln\left(1+ \frac{k}{k-1}\right)}{k} + \frac{\sum\limits_{j=1}^{k-1} \frac{\ln\left(1+ \frac{j}{k-1}\right)}{\frac{j}{k-1}}}{k-1}\\
&= \int\limits_0^1 \frac{\ln(1+x)}{x} + O(k^{-1})\\
&= \frac{\pi^2}{12} + O(k^{-1})
\end{align*}

Since $e^{\frac{\pi^2}{12 k} + O(k^{-2})} = 1 + \frac{\pi^2}{12 k} +  O(k^{-2})$, we have that $\frac{\gamma_k}{k} = 1 - \frac{1 - \frac{\pi^2}{12}}{k} + O(k^{-2})$. Since $\gamma_k= c_k^2$, we can use the Taylor series decomposition of $\sqrt{1+x}$ to see that $\frac{c_k}{\sqrt{k}} = \sqrt{1 - \frac{t_2}{k} + O(k^{-2})} = 1 - \frac{t_2}{2k} + O(k^{-2})$.
\end{proof}

The fact that $t_2 > t_1$ implies that $c_k \leq c_{q,k}$ for large enough $q$, but we can actually prove more:

\begin{theorem}
Any decomposition bound that decomposes with this greedy approach with restricted block sizes gives a bound that is no better than this greedy algorithm. In particular, $c_{q,k} \geq c_k$ for any $q \leq k$.
\end{theorem}

\begin{proof}
If block size $i$ is not allowed, then the LP corresponding to Equation \ref{old-primal} for that algorithm has $a_i = 0$ and the inequality corresponding to $i$ is omitted. Taking the dual of this new LP gives an LP whose feasible region is contained in the feasible region of Equation \ref{old-dual}, except we force that $b_i = 0$.

This means that the optimal solution of the modified primal problem is the same value as a feasible dual solution, and thus is at least as large as $\alpha$. This means that the induced bound on $M(n,k)^2$ is larger with the modified algorithm.

Scheinerman's methods \cite{Scheinerman:2019aa} for $c_{q,k}$ are simply modified greedy algorithms, where only blocks of size 1 and $q$  are used (only $a_1$ and $a_q$ are non-zero). Therefore the unmodified greedy algorithm bound $c_k \leq c_{q,k}$ for any $k$ and $q$.

\end{proof}
\subsection{Improving the greedy algorithm}

To attempt to do better, we try to utilize the idea that we can take all the rows with a one in the chosen column, meaning that column cannot be chosen in future iterations. In fact, since the remainder of the chosen column is guaranteed to have no ones, we can remove that column for all future iterations of our algorithm without affecting the determinant.

This means that after we have removed all blocks of size at least $i$, there are still $m = n - \sum_{j=i}^n jx_j$ rows, but now only $n - \sum_{j=i}^n x_j$ columns, meaning that we can get more large blocks. This means that the average number of ones per row is $k\frac{n - \sum_{j=i}^n jx_j}{n - \sum_{j=i}^n x_j}$ which again must be no larger than $i-1$. Rearranging this gives us the constraints in the following linear program:

\begin{lemma}\label{new-proof}
For any fixed $k$, let $\alpha$ be the optimal solution of the following LP:

\begin{equation}\label{new-primal}
\begin{array}{ll@{}ll}
\text{maximize}  & \alpha' = \sum\limits_{j=1}^{n} \ln\left(1+ \frac{j}{k-1}\right)a_j' &\\
\text{subject to}& \sum\limits_{j=i}^n (kj-i+1)a_j' \geq k-i+1, \hspace{.5cm}& i=2 ,..., k\\
& \sum\limits_{j=1}^n ja_j' = 1\\
& a_{j} \geq 0 ,& j=1 ,..., n
\end{array}
\end{equation}
Then $M(n,k)^2 \leq (\gamma_k')^{n}$ where $\gamma_k' = (k-1)e^{\alpha'}$.
\end{lemma}

We can also find the dual in the same manner as before:

\begin{equation}\label{new-dual}
\begin{array}{ll@{}ll}
\text{minimize}  & \beta' = \sum\limits_{i=1}^{k} (k-i+1)b_i' &\\
\text{subject to}& \sum\limits_{i=1}^j (kj-i+1) b_i' \geq \ln\left(1+ \frac{j}{k-1}\right), \hspace{.5cm}& j=1,..., n\\
& b_i' \leq 0 ,&i=2 ,..., k
\end{array}
\end{equation}

While the analog of Lemma \ref{old-calculation} holds again, the proof is a little more involved:

\begin{lemma}\label{new-calculation}
The optimal solution to equation \ref{new-primal} is given by $a_j' = 0$ for $j > k$ and the other $a_j'$ determined by making all $k$ non-trivial constraints tight. The optimal solution to equation \ref{new-dual} is given by the $b_i'$ that make the first $k$ constraints tight.
\end{lemma}

\begin{proof}
To show feasibility of the primal solution, we need to show that each $a_j$ is positive, as the other constraints are already tight. This is trivial for each $j > k$, so we focus on the case where $j \leq k$.

To do this, we note that $\sum_{j=i}^k (kj-i+1) a_j' = k-i+1$ and $\sum_{j=i+1}^k (kj-i) a_j' = k-i$. When solving for $a_i'$, we get the recursive definition: $a_i' = \frac{1-\sum_{j=i+1}^k a_j'}{(k-1)i+1}$. The exact values can be computed from this, but this is unnecessary to show that $a_i' \geq 0$.

We can simply note that this means that each $a_i'$ is a small fraction of the distance between $\sum_{j=i+1}^k a_j'$ and $1$. By induction down from $k$, this means that every $\sum_{j=i+1}^k a_j' \leq n$, and thus $a_j' \geq 0$ for every $j$.

In addition, we can write the non-zero $a_j'$ in the same manner as in Lemma \ref{old-calculation}, so $\vec{a'} = (M')^{-1}\vec{v}$ and $\alpha = \vec{c}^\top (M')^{-1}\vec{v}$, where all quantities are the same as in Lemma \ref{old-calculation} except that $M'_{ij} = kj - i +1$ for $j \geq i$ instead.

To show feasibility of the dual solution, we need to show that $b_i' \leq 0$ for each $i \geq 2$ and also that $\sum\limits_{i=1}^j (kj-i+1) b_i' \geq \ln\left(1+ \frac{j}{k-1}\right)$ for each $j > k$, as the other constraints are already tight.

To show both of these things, we rely heavily on the following equation:
\begin{equation}\label{proof-helper}
(j-1)\sum_{i=1}^{j} (kj - i + 1)b_i' - j\sum_{i=1}^{j-1} (k(j-1) - i + 1)b_i' = (j-1)[(k-1)(j-2) +1]b_j' + \sum_{i=1}^{j-1} (i-1) b_i'
\end{equation}

To show that $b_j' \leq 0$ for $2 \leq j \leq k$, we simply notice that $\sum_{i=1}^{j} (kj - i + 1)b_i' = \ln\left(1+ \frac{j}{k-1}\right)$ and $\sum_{i=1}^{j-1} (k(j-1) - i + 1)b_i' = \ln\left(1+ \frac{j-1}{k-1}\right)$ because $ j\leq k$. This means that the left hand side is equal to $(j-1)\ln\left(1+ \frac{j}{k-1}\right) - j\ln\left(1+ \frac{j-1}{k-1}\right)$, so in order to show $b_j'$ is negative, we simply need to show that $ \sum_{i=1}^{j-1} (i-1) b_i' \geq (j-1)\ln\left(1+ \frac{j}{k-1}\right) - j\ln\left(1+ \frac{j-1}{k-1}\right)$ for each $j \geq 2$.

We show this by induction on $j$. The base case is when $j = 2$, where the left hand side is 0. A simple application of Jensen's inequality gives that the right hand side is negative, and thus $b_2' \leq 0$.

Now inductively assume that $\sum_{i=1}^{j-1} (i-1) b_i' \geq (j-1)\ln\left(1+ \frac{j}{k-1}\right) - j\ln\left(1+ \frac{j-1}{k-1}\right)$, and thus that $b_j' \leq 0$. Using Jensen, we have that:
\begin{multline}\label{proof-helper2}
\left[j\ln\left(1 + \frac{j+1}{k-1}\right) - (j+1)\ln\left(1 + \frac{j}{k-1}\right)\right] - \left[(j+1)\ln\left(1 + \frac{j+2}{k-1}\right) - (j+2)\ln\left(1 + \frac{j+1}{k-1}\right)\right]\\
= (j+1)[2\ln\left(1 + \frac{j+1}{k-1}\right) - \ln\left(1 + \frac{j}{k-1}\right) - \ln\left(1 + \frac{j+2}{k-1}\right)] > 0
\end{multline}

This means $(j-1)[(k-1)(j-2) +1]b_{j}' \leq (j-1)b_j'$, so using our constraints we get that $\sum_{i=1}^{j} (i-1) b_i' \geq (j-1)[(k-1)(j-2) +1]b_{j}' + \sum_{i=1}^{j-1} (i-1) b_i'$. By equation \ref{proof-helper2}, we have that $j\ln\left(1 + \frac{j+1}{k-1}\right) - (j+1)\ln\left(1 + \frac{j}{k-1}\right) \geq (j-1)\ln\left(1+ \frac{j}{k-1}\right) - j\ln\left(1+ \frac{j-1}{k-1}\right)$, so we have our induction.

This gives that $b_j' \leq 0$ for $j \geq 2$. To do the other constraints, we simply leverage equations \ref{proof-helper} and \ref{proof-helper2} differently.

We prove that $\sum_{i=1}^j (kj - i + 1)b_i' \geq \ln\left(1 + \frac{j}{k-1}\right)$ inductively, where the base case is the given equality when $j = k$. Then, assuming truth for $j$, the inductive step is simply:
\begin{align*}
\sum_{i=1}^{j+1} (k(j+1) - i + 1)b_i' &= \frac{j[(k-1)(j-1) +1]b_{j+1}' + \sum_{i=1}^j (i-1) b_i' + (j+1) \sum_{i=1}^{j} (kj - i + 1)b_i'}{j}\\
&\geq \frac{ j\ln\left(1 + \frac{j+1}{k-1}\right) - (j+1)\ln\left(1 + \frac{j}{k-1}\right) +  (j+1)\ln\left(1 + \frac{j}{k-1}\right)}{j}\\
& =  \ln\left(1 + \frac{j+1}{k-1}\right)
\end{align*}

This means that the $b_i'$ are a feasible solution to the dual, and so like in Lemma \ref{old-calculation}, we have $\vec{b'} = (M'^\top)^{-1} \vec{c}$ and $\beta' = \vec{v}^\top\vec{b'} = \vec{v}^\top((M')^{-1})^\top \vec{c}$. Since $\alpha' = \beta'$, both are optimal.
\end{proof}

As in Theorem \ref{old-main}, we now can say

\begin{theorem}\label{new-main}
For every $k$, there is a $c_k'$ such that $M(n,k) \leq c_k'$ where $c_k' = \sqrt{k} - \frac{t_3}{2\sqrt{k}} + O(k^\frac{-3}{2})$ for $t_3 \approx 0.178$.
\end{theorem}

\begin{proof}
Using Lemma \ref{new-proof} and Lemma \ref{new-calculation}, we have that $M(n,k) \leq (c_k')^n$ for $c_k' = \sqrt{(k-1)e^{\alpha'}}$.
\end{proof}

Experimentally $t_3 \approx 0.178$, but the recursive nature of the computations for finding the $a_i$ (shown in the proof of Lemma \ref{new-calculation}) makes explicit calculation of $t_3$ difficult.
While this seems identical to Theorem \ref{old-main}, we can prove that this new $c_k'$ is a strict improvement on the $c_k$ from the previous analysis.

\begin{theorem}
The optimal solution $\alpha'$ to equation \ref{new-primal} is strictly smaller than $\alpha$ in equation \ref{old-primal}. Therefore, the improved algorithm gives a strictly better bound on $M(n,k)$. 
\end{theorem}

\begin{proof}
Since $\alpha = \sum_{j=1}^k a_j\ln\left(1 + \frac{j}{k-1}\right)$, the difference is $\sum_{j=1}^k \ln\left(1 + \frac{j}{k-1}\right)\left(a_j - \frac{1}{jk} \right)$. Rearranging the sums gives that this difference is:

\begin{equation}
\ln\left(1 + \frac{1}{k-1}\right)\left(\sum_{j=1 }^k ja_j' - 1\right) + \sum_{i = 2}^k \frac{\ln\left(1 + \frac{i}{k-1}\right)}{i} - \frac{\ln\left(1 + \frac{i-1}{k-1}\right)}{i-1}\left(\sum_{j=i}^k ja_j' - \frac{k-i+1}{k}\right)
\end{equation}

Because $\sum_{j=1 }^k ja_j' = 1$, the first term is simply 0. Using the other constraints in equation \ref{new-primal}, we know that $\sum_{j=i}^k (kj - i + 1)a_j' = k-i+1$. This means that $\sum_{j=i}^k ja_j' - \frac{k-i+1}{k} = \frac{i-1}{k}\sum_{j=i}^k a_j' > 0$ for each $2 \leq i \leq k$.

Using Jensen's inequality, we can see that $\frac{\ln\left(1 + \frac{i}{k-1}\right)}{i} - \frac{\ln\left(1 + \frac{i-1}{k-1}\right)}{i-1} < 0$ for every $i \geq 2$, so every other term is negative. This means that $\alpha' < \alpha$ , and so $c_k' < c_k$ is an improvement on our earlier algorithm.
\end{proof}

Experimentally, this difference appears to be $O(k^{-2})$, and thus has no effect on the asymptotic convergence in Theorem \ref{old-main}, meaning that $t_2 = t_3$.

\section{Summary of results}

While it is impossible to improve on Ryser's bound when $\lambda = \frac{k(k-1)}{n-1}$ is $\Omega(1)$ ($k = \Omega(\sqrt{n})$), for fixed $k$ we have an improvement on the previously known bounds as shown below (in order of decreasing bound):

\begin{table}[ht]
\caption{Known bounds on $M(n,k)$}\label{table-1}
\renewcommand\arraystretch{1.5}
\noindent\[
\begin{array}{|c|c|c|}
\hline
& M(n,k)\text{ bound}&\text{Asymptotic bound on behavior of }\limsup_{n} M(n,k)^{\frac{1}{n}}\\
\hline
\text{Hadamard \cite{Had_1893}} & k^\frac{n}{2} & \sqrt{k}\\
\hline
\text{Ryser \cite{Ryser_1956}} & k(k-\lambda)^\frac{n-1}{2} & \sqrt{k}\\
\hline
\text{Schienerman \cite{Scheinerman:2019aa}} & (c_{2,k})^n & \sqrt{k} -\frac{1}{4k^\frac{3}{2}}+ O(k^{-3}) \\
\hline
\text{Schienerman \cite{Scheinerman:2019aa}} & (c_{q_k,k})^n & \sqrt{k} - \frac{.1}{2\sqrt{k}} + O(k^\frac{-3}{2}) \\
\hline
\text{Theorem } \ref{old-main} & (c_k)^n & \sqrt{k} - \frac{.18}{2\sqrt{k}} + O(k^\frac{-3}{2})\\
\hline
\text{Theorem } \ref{new-main} & (c'_k)^n & \sqrt{k} - \frac{.18}{2\sqrt{k}} + O(k^\frac{-3}{2}) \footnotemark \\
\hline
\end{array}
\]
\end{table}

\footnotetext{Constant is conjectured}

None of these approaches achieve the known lower bound given by a block diagonal matrix of projective planes (asymptotically $\sqrt{k} - \frac{1}{2\sqrt{k}} + O(k^\frac{-3}{2})$), but the gap is closing. 

\subsection{Limitations and avenues for improvement}
Note that these bounds have three places where there can be loss:

\begin{enumerate}
    \item In Fischer's inequality for bounding the volume when subdividing into blocks.
    \item In the determinant estimation within a single block.
    \item In the LP relaxation of the integer program.
\end{enumerate}

Since all constraints are integral in equations \ref{old-primal} and \ref{new-primal}, the optimal $a_j$ are always rational. This means that if we choose an $n$ that is a multiple of the least common denominator, the LP solution is then an integral solution. For instance, if we consider \ref{new-primal} with $k = 3$, using $n = 35$ means we can decompose the rows into 5 blocks of 3, 6 blocks of 2, and 8 blocks of 1. For $k = 4$, we can use $n=455$ for the same result.

The determinant estimation is tight within a single block $A_i$ when $A_iA_i^\top = J_{m_i} + (k-1)I_{m_i}$. Because $A_iA_i^\top$ is a Gram matrix, and there is a column of $A_i$ that is all 1, we simply need to place the other $k-1$ ones in each row such that no two are in the same column.

In the case where \ref{new-primal} is tight that is described above, we have that $\sum\limits_{j=i}^n (kj-i + 1)x_j = (k-i+1)n$ for each $1 \leq i \leq k$. If we consider all the rows in all the $x_i$ blocks of size $i$, there are $ix_i$ rows where we need to place $k-1$ ones in each. Since none of the chosen rows can have any more ones in them there are $\sum_{j=i}^n x_j$ rows we can't use. As in the proof of Lemma \ref{new-calculation}, we know that $a_i = \frac{1-\sum_{j=i+1}^k a_j}{(k-1)i+1}$. Rescaling by $n$ gives that $n-\sum_{j=i}^k x_j = ix_i(k-1)$, so we only need to put a single one in each column for all $x_i$ blocks!

This means that our determinant estimation can be tight inside every block. In addition, this means that if we have two blocks of size $i$, if we take a vector from each block, they'll always be perpendicular. This means that $B_1B_2= 0$ in Fischer's inequality for these blocks as well, so our bound is tight there.

Since each set of blocks puts a one in each row that isn't chosen, clearly this part of Fischer's inequality isn't tight when the blocks are different sizes. Looking closer shows that each column also has $k$ non-zero entries in this example.

This suggests that the first type of error is where any improvement could be made to this bound, as opposed to the other two. However, it is likely a more nuanced approach will need to be used to balance these different errors.

To see this, we randomly generated many examples of this form for $k=3$, and found that the maximum of their determinants never exceeded the asymptotic lower bound given by the Fano plane (which is $24^\frac{n}{7}$).

This suggests that perhaps the projective plane (or a block matrix of them) has the (asymptotically) largest determinant of any matrix in $T(n,k)$. When examining the greedy algorithm used on a projective plane, we can see that the decomposition is into 1 block of size $k$ and the remainder into $k-1$ blocks of size $k-1$. Only the $i=k$ inequality is tight here, and so the projective plane is nowhere close to tight with the LP bound.

However, while you can see that this block decomposition still has that the determinant estimation is tight within each block, but there still is significant error coming from the volume submultiplicativity. Even if we preset $x_j$ to the actual number of blocks, the volume bound we get for the projective plane of with $k$ ones in each row (order $k-1$) is $(k-1)^{k^2-k+1} \frac{2k-1}{k-1}\left(\frac{2k-2}{k-1}\right)^{k-1} = (2k-1)2^{k-1}(k-1)^{k^2-k-1}$. This is much larger than the actual volume which is $k^2(k-1)^{k^2-k}$, so our this loss is about a factor of $\frac{2^k}{k^2}$.

Both of these methods of generating large determinant matrices where the row sums are $k$ actually generate matrices where the column sums are also $k$. This suggests that if we limit $A$ such that both $A$ and $A^\top$ are both in $R(n,k)$, we should get the same maximum determinants. However, unless we have specific $n$ and $k$ such that Ryser's bound is tight, this is unproven as of yet.

\section*{Acknowledgements}

Thanks to Daniel Scheinerman for introducing me to this problem, as well as the many valuable discussions we had. I would also like to thank my colleagues Keith Frankston and Yonah Biers-Ariel for reading the paper. And lastly, my advisor Swastik Kopparty for doing all of the above.

\printbibliography

\end{document}